\documentclass[11pt]{article}
\usepackage{amsmath,amssymb,amsfonts,amsthm}
\usepackage{float}
\numberwithin{equation}{section}
\newtheorem{theorem}{Theorem}[section]
\newtheorem{definition}{Definition}[section]
\newtheorem{lemma}[theorem]{Lemma}

\newtheorem{corollary}[theorem]{Corollary}

\begin{document}
\begin{center}
{\Large{\textbf{{COLORING SUMS OF EXTENSIONS OF CERTAIN GRAPHS}}}} 
\end{center}
\vspace{0.5cm}
\centerline{\large{Johan Kok}} 
\centerline{\small{Tshwane Metropolitan Police Department}}
\centerline{\small{City of Tshwane, Republic of South Africa}}
\centerline{\tt {kokkiek2@tshwane.gov.za}}
\vspace{0.5cm}
\centerline{\large{Saptarshi Bej}} 
\centerline{\small{Department of Mathematics and Statistics}}
\centerline{\small{Indian Institute of Science Education}}
\centerline{\small{Kolkata, Mohanpur-741246, India}}
\centerline{\tt {saptarshibej24@gmail.com}}
\vspace{0.5cm}
\begin{abstract}
\noindent We recall that the minimum number of colors that allow a proper coloring of graph $G$ is called the chromatic number of $G$ and denoted $\chi(G)$. Motivated by the introduction of the concept of the $b$-chromatic sum of a graph by Lisna and Sunitha, the concept of $\chi'$-chromatic sum and $\chi^+$-chromatic sum are introduced in this paper. The extended graph $G^x$ of a graph $G$ was recently introduced for certain regular graphs. This paper furthers the concepts of $\chi'$-chromatic sum and $\chi^+$-chromatic sum to extended paths and cycles. Bipartite graphs also receive some attention. The paper concludes with \emph{patterned structured} graphs. These last said graphs are typically found in chemical and biological structures.
\end{abstract}
\noindent {\footnotesize \textbf{AMS Classification Numbers:} 05C05, 05C20, 05C38, 05C62} 
\section{Introduction}
For general notation and concepts in graph and digraph theory, we refer to [2, 4, 5]. Unless mentioned otherwise, all graphs mentioned in this paper are simple, connected, finite and undirected graphs of order $n \geq 2$. We recall that the minimum number of colors that allow a proper coloring of graph $G$ is called the chromatic number of $G$ and denoted $\chi(G)$. Since coloring may be effected from any alphabet like in an application of cryptographical coding of vertices, we use a preferred definition slightly different from the contemporary definition found in the literature. Consider a \emph{proper k-coloring} of a graph $G$ and denote the set of $k$ colors, $\mathcal{C} = \{c_1,c_2, c_3, \dots , c_k\}.$ Also consider the disjoint subsets of $V(G)$ i.e. $V_{c_i} = \{v_j: v_j \rightarrow c_i, v_j \in V(G), c_i \in \mathcal{C}\},$ $1 \leq i \leq k.$ Clearly, $V(G) = \bigcup\limits_{i=1}^{k}V_{c_i}.$ If for largest $k \in \Bbb N$ a proper k-coloring is found such that there exist $v_s \in V_{c_i}$ and $v_t \in V_{c_j},$ such that edge $v_sv_t \in E(G)$ for all distinct pairs of colors $c_i, c_j$ then the $b$-chromatic number of $G$ is defined to be $\varphi(G) = k.$ Such a coloring is called a $b$-coloring of $G$.\\\\
In a paper by Anirban Banerjee and Saptarshi Bej [1] the concept of extending regular graphs has been introduced. We shall use the idea of extending a graph (also called an extended graph) in general and study how certain interesting invariants change in the extended graph.
\section{On Extensions of Graphs}
We formally define the notion of a \emph{degree-extension} of a graph $G$.
\begin{definition}
(i) Let $G$ be a graph of even order $n\geq 2$. If the complement graph $G^c$ has a perfect matching say $M,$ then $G^x_c = G+M$ is called a complete degree-extension of $G$, also called an $c$-extended $G$.\\
(ii)  Let $G$ be a graph of odd order $n \geq 3$ and let $u \in V(G)$ such that $d_G(u) = \Delta(G)$.  If the complement graph $(G-u)^c$ has a perfect matching say $M$, then $G^x_{ic} = G+ M$ is called an incomplete degree-extension of $G$, also called an $ic$-extended $G$.
\end{definition}
Note that the choice of vertex $u$, $d_G(u) = \Delta(G)$ in Definition 2.1(ii) is by default. Other criteria can apply of which $u$, $d_G(u) = \delta(G)$ is an obvious other criteria. An immediate and important result following from Definition 2.1 is given by the next corollary.
\begin{corollary}
From Definition 2.1 it follows that:\\
(i) In $G^x_c$, $d_{G^x_c}(v) = d_G(v) + 1$, $\forall v \in V(G).$\\
(ii) In $G^x_{ic}$, $d_{G^x_{ic}}(v) = d_G(v) + 1$, $\forall v \in V(G),$ $v\neq u$ and $\Delta(G^x_{ic}) \in \{\Delta(G), \Delta(G) + 1\}.$
\end{corollary}
In [1] it is shown that a $r$-regular graph $G$ of even order $n$ with $r < \frac{n}{2}$ always has a degree-extension $G^x_c$.  It means that $G$ can always be extended to a $(r+1)$-regular graph $G^x_c$ by adding edges to $G$. First, we generalise the result to any simple connected graph $G$ for which $\Delta(G) < \frac{n}{2}.$ 
\begin{theorem}
Let $G$ be a graph of order $n$ with $\Delta(G) < \frac{n}{2},$ then graph $G$ has a complete degree-extension $G^x_c$.
\end{theorem}
\begin{proof}
Following from Definition 2.1 we consider two cases.\\
Case (i): See [1]. Assume $n$ is even. Since $\Delta(G) < \frac{n}{2}$ the complement graph $G^c$ has $\delta(G^c) \geq \frac{n}{2}.$ Recalling Dirac's Theorem it follows that $G^c$ is Hamiltonian. So for each Hamilton cycle in $G^c$ exactly two perfect matchings exist in $G^c$. Therefore, $G$ has at least two complete degree-extensions, $G^x_c$.\\
Case (ii): Assume $n$ is odd. There exists a vertex $u \in V(G)$ such that $d_G(u) = \Delta(G).$ Consider the graph $G-u$ then certainly $\Delta(G-u) < \frac{n-1}{2}$. Following from Case (i) we have that $G-u$ has a degree extension $(G-u)^x_c.$ From Definition 2.1(ii) it follows that $(G-u)^x + u$ is an incomplete degree-extension $G^x_{ic}$ of $G$, with $\Delta(G^x_{ic}) \in \{\Delta(G), \Delta(G) + 1 \}.$
\end{proof}
It is observed that $\lfloor\frac{n}{2}\rfloor$ edges must be added to obtain an incomplete degree-extension $G^x_{ic}$. \textbf{Henceforth we will only consider a graph $G$ of even order $n$.}\\\\
If $n$ is even and if $k$ distinct Hamilton cycles exist in $G^c$ then $2k$ complete extended graphs exist. For graphs on even number of vertices we have the next result.
\begin{theorem}
Let $G$ be a graph of even order $n$. If the complement graph $G^c$ has a path of length $n-1$ (spanning path) the graph $G$  has a complete degree-extension $G^x_c$. 
\end{theorem}
\begin{proof}
If the complement graph $G^c$ has a spanning  path of length $n-1$ exactly one perfect matching corresponding to that path in $G^c$ exists. The result follows immediately.
\end{proof}
\begin{corollary}
Let $G$ be a graph of even order $n$. If $V(G)$ can be partitioned into a number of subsets $V_1(G), V_2(G),\dots , V_\ell(G)$, each containing an even number of vertices such that the complement of each induced subgraph $\langle V_i(G)\rangle ^c$, $1 \leq i \leq \ell$ has a spanning path, then $G$ has a complete degree-extension $G^x_c$. 
\end{corollary}
\begin{proof}
Denote the spanning path in $\langle V_i(G)\rangle^c$ by$P^{(i)}$, $1 \leq i \leq \ell$. From each $P^{(i)}$ we select the unique perfect matching $S^{(i)}$ in $\langle V_i(G)\rangle^c.$ Clearly $\bigcup\limits_{i=1}^{\ell}S^{(i)}$ is a perfect matching in $G^c$.
\end{proof}
Clearly the smallest path for which Theorem 2.2 finds application is $P_4$. Let the vertices of $P_4$ consecutively be labeled $v_1,v_2,v_3,v_4.$ Then the unique $P^x_{4,c} = P_4 + (v_1v_3, v_2v_4).$ For $P_6$ we find the application of Theorem 2.2 allows $P^{x_1}_{6,c}= P_6 + (v_1v_6, v_2v_4, v_3v_5)$ or $P^{x_2}_{6,c}= P_6 + (v_1v_3, v_2v_5, v_4v_6)$ or $P^{x_3}_{6,c}= P_6 + (v_1v_4, v_2v_5, v_3v_6)$ or $P^{x_4}_{6,c}= P_6 + (v_1v_5, v_2v_4, v_3v_6)$ or $P^{x_5}_{6,c}= P_6 + (v_1v_4, v_2v_6, v_3v_5).$ 
\section{$\chi'$-Chromatic Sum and $\chi^+$-Chromatic Sum of Certain Graphs}
We recall that a vertex coloring of a graph $G$ such that adjacent vertices are not allocated the same color is called a proper coloring of $G$. The minimum number of colors in a proper coloring of $G$ is called the chromatic number, $\chi(G).$
\subsection{The $\chi'$-Chromatic Sum and $\chi^+$-Chromatic sum of Extended Paths and Extended Cycles}
In a paper [9],  Lisna and Sunitha introduce a new concept called the b-Chromatic Sum of a Graph. Let $\mathcal{C} = \{c_1,c_2, c_3, \dots , c_k\}$ allow a $b$-coloring $\mathcal{S}$ of $G$. As stated in [7] there are $k!$ ways of allocating the colors to the vertices of $G$. Let the \emph{color weight} $\theta(c_i)$ be the number of times a color $c_i$ is allocated to vertices. In general we refer to the \emph{color sum} of a coloring $\mathcal{S}$ and define it, $\omega(\mathcal{S}) = \sum\limits_{i=1}^{k}i\cdot\theta(c_i).$ The $b$-chromatic sum define by Lisna and Sunitha is given by $\varphi'(G) = min\{\sum\limits_{i=1}^{k}i\cdot\theta(c_i): \forall$ $b$-colorings of $G\}.$ This interesting new invariant motivates similar concepts in graph coloring. Also see [7]. 
\begin{definition}$[7]$
For a graph $G$ the $\chi'$-chromatic sum is defined to be:\\ $\chi'(G) = min\{\sum\limits_{i=1}^{k}i\cdot\theta(c_i): \forall$  minimum proper colorings of $G\}.$
\end{definition} 
\begin{definition}$[7]$
For a graph $G$ the $\chi^+$-chromatic sum is defined to be:\\ $\chi^+(G) = max\{\sum\limits_{i=1}^{k}i\cdot\theta(c_i): \forall$  minimum proper colorings of $G\}.$
\end{definition}
Further motivation for these new invariants is as follows. If the colors represent different technology types and the configuration requirement is that at least one unit per technology type must be placed at a point in a network without similar technology types being adjacent, two further considerations come into play. Firstly, if the higher indexed colors represent technology types with higher failure rate (risk) then the placement of the maximal number of higher indexed units is the solution to ensure a functional network. On the other hand is the lower indexed colors represent a less costly (procurement, installation, commissioning and maintenance) technology type, and minimising total cost is the priority, then the placement of maximal number of lower indexed units is the desired solution. We recall two important results from [7].
\begin{theorem}$[7]$
For a path $P_n$, $n \geq 1$ the $\chi'$-chromatic sum and $\chi^+$-chromatic sum are given by:\\
(i)
\begin{equation*} 
\chi'(P_n)) =
\begin{cases}
1, &\text {if $n =1$,}\\  
\frac{3n}{2}, &\text {if $n$ is even},\\
3\cdot\lfloor\frac{n}{2}\rfloor + 1, &\text {if $n$ is odd}.
\end{cases}
\end{equation*}\\
(ii)
\begin{equation*} 
\chi^+(P_n)) =
\begin{cases}
1, &\text {if $n =1$,}\\  
\frac{3n}{2}, &\text {if $n$ is even},\\
3\cdot\lfloor\frac{n}{2}\rfloor + 2, &\text {if $n$ is odd}.
\end{cases}
\end{equation*}  
\end{theorem}
\begin{theorem}$[7]$
For a cycle $C_n$ the $\chi'$-chromatic sum and $\chi^+$-chromatic sum are given by:\\
(i) \begin{center}$\chi'(C_n)) = 3\cdot\lceil\frac{n}{2}\rceil.$ \end{center}
(ii) \begin{equation*} 
\chi^+(C_n)) =
\begin{cases}
\frac{5n}{2}, &\text {if $n$ is even},\\
5\cdot\lfloor\frac{n}{2}\rfloor + 1, &\text {if $n$ is odd}.
\end{cases}
\end{equation*}   
\end{theorem}
From $P^{x_i}_{6,c}$, $1\leq i \leq 5$ we note that $P^{x_3}_{6,c}$ allows a minimum proper coloring $v_1\rightarrow c_1, v_2\rightarrow c_2, v_3\rightarrow c_1, v_4\rightarrow c_2, v_5\rightarrow c_1, v_6\rightarrow c_2,$ such that $\chi'(P^{x_3}_{6,c}) = 9 = min\{\chi'(P^{x_i}_6): 1\leq i \leq 5\}.$ Also we note that $P^{x_4}_6$, (by symmetry $P^{x_5}_6$ as well) allows a minimum proper coloring $v_1\rightarrow c_3, v_2\rightarrow c_2, v_3\rightarrow c_1, v_4\rightarrow c_3, v_5\rightarrow c_2, v_6\rightarrow c_3,$ such that $\chi^+(P^{x_4}_{6,c}) = 14 = max\{\chi^+(P^{x_i}_{6,c}): 1\leq i \leq 5\}.$ Up to isomorphism $P_6$ has 4 distinct complete extensions. These observations motivate the next definitions.
\begin{definition}
 For a graph $G$ of even order $n,$ such that the complement graph $G^c$ has at least one perfect matching then, if, up to isomorphism the graph $G$ has $G^{x_i}_c$, $1\leq i \leq \ell$  complete degree-extensions, then $\chi'(G^x_c) = min\{\chi'(G^{x_i}_c): 1\leq i \leq \ell\}.$
\end{definition}
\begin{definition}
 For a graph $G$ of even order $n,$ such that the complement graph $G^c$ has at least one perfect matching then, if up to isomorphism the graph $G$ has $G^{x_i}_c$, $1\leq i \leq \ell$ complete degree-extensions, then $\chi^+(G^x_c) = max\{\chi^+(G^{x_i}_c): 1\leq i \leq \ell\}.$
\end{definition}
\begin{theorem}
For a path $_n$, $n\geq 4$ and even, we have:\\
(i) $\chi'(P^x_{4,c}) = 7$ and $\chi^+(P^x_{4,c}) = 9$.\\
(ii) $\chi'(P^x_{n,c}) = \frac{3n}{2}$ and $\chi^+(P^x_{n,c}) = \frac{5n}{2}-1$, for $n\geq 6.$
\end{theorem}
\begin{proof}
Case (i): Consider the unique $P^x_{4,c} = P_4 + (v_1v_3, v_2v_4)$ and allocate the minimum proper coloring $v_1\rightarrow c_1, v_2\rightarrow c_2, v_3\rightarrow c_3, v_4\rightarrow c_1$. It follows that $\chi'(P^x_{4,c}) = 7.$ Now interchange the colors $c_1$ and $c_3$. It follows that $\chi^+(P^x_{4,c}) = 9.$\\
Case (ii)(a): Consider $P_n$, $n\geq 6.$\\
Subcase (ii)(a)(1): If $\frac{n}{2}$ is odd, consider $P^x_{n,c} = P_n + (v_1v_{\frac{n}{2} + 1}, v_{\frac{n}{2}}v_n, v_{1+i}v_{n-i})$, $1\leq i \leq \frac{n}{2}-1$. Allocate the minimum proper coloring $v_1\rightarrow c_1, v_2\rightarrow c_2, v_3\rightarrow c_1,\dots, v_n\rightarrow c_2.$ So, $\theta(c_1) = \frac{n}{2}$ and $\theta(c_2) = \frac{n}{2}.$ Hence $\chi'(P^x_{n,c}) = 1\cdot\theta(c_1) + 2\cdot\theta(c_2) = \frac{n}{2} + 2\cdot\frac{n}{2} = \frac{3n}{2} = min\{\chi'(P^{x_i}_{n,c}): \forall P^{x_i}_{n,c}\}.$\\
Subcase (ii)(a)(2): If $\frac{n}{2}$ is even, consider $P^x_{n,c} = P_n + (v_1v_{\frac{n}{2}}, v_{\frac{n}{2}+1}v_n, v_{1+i}v_{n-i})$, $1\leq i \leq \frac{n}{2}-1$. The result now follows like in Subcase (ii)(a)(1).\\
Case (ii)(b): Consider $P_n$, $n\geq 6.$\\
Subcase (ii)(b)(1): If $\frac{n}{2}$ is odd let $t=\frac{n}{2}.$ Construct the graph $P^*_n = P_n + (v_1v_{t+1},v_{t-1}v_{t+3},v_tv_{t+2})$. Now allocate the proper coloring $v_1\rightarrow c_3, v_2\rightarrow c_2, v_3\rightarrow c_3,\dots, v_{t-1}\rightarrow c_2, v_t \rightarrow c_3, v_{t+1}\rightarrow c_1, v_{t+2}\rightarrow c_2, v_{t+3}\rightarrow c_3,\dots, v_{n-1}\rightarrow c_2, v_n\rightarrow c_3.$ Since equal number of vertices are left with degree 2 and color $c_3$,$c_2$ respectively it is always possible to complete the extension to obtain $P^x_{n,c}$ having the allocated minimum proper coloring. Clearly $\theta(c_1) =1,$ $\theta(c_2) = t-1$ and $\theta(c_3) = t.$ Hence, $\chi^+(P^x_{n,c}) = 3\cdot\frac{n}{2} + 2\cdot(\frac{n}{2} -1) + 1 = \frac{5n}{2} -1 = max\{\chi^+(P^{x_i}_{n,c}):\forall P^{x_i}_{n,c}\}.$\\
Subcase (ii)(b)(2): The result follows similar to Subcase (ii)(b)(1).
\end{proof}
\begin{theorem}
For a cycle $C_n$, $n = 2t \geq 4$, we have:\\
(i) $\chi'(C^x_{4,c}) =\chi^+(C^x_{4,c}) = 10$.\\
(ii) $\chi'(C^x_{n,c}) = \frac{3n}{2}$ and: 
\begin{equation*} 
\chi^+(C^x_{n,c})) =
\begin{cases}
\frac{5n}{2}-1, &\text {if $t$ is even, $n\geq 6$},\\
\frac{5n}{2} -3, &\text {if $t$ is odd, $n\geq 6$}.
\end{cases}
\end{equation*}   
\end{theorem}
\begin{proof}
Case (i): Consider the unique $C^x_{4,c} = C_4 + (v_1v_3, v_2v_4) = K_4$.  Hence, $\chi^+(C^x_{4,c}) = 10.$\\
Case (ii)(a): Consider $C_n$, $n\geq 6.$\\
Subcase (ii)(a)(1): If $\frac{n}{2}$ is odd, consider $C^x_{n,c} = C_n + (v_1v_{\frac{n}{2} + 1}, v_{\frac{n}{2}}v_n, v_{1+i}v_{n-i})$, $1\leq i \leq \frac{n}{2}-1$. Allocate the minimum proper coloring $v_1\rightarrow c_1, v_2\rightarrow c_2, v_3\rightarrow c_1,\dots, v_n\rightarrow c_2.$ So, $\theta(c_1) = \frac{n}{2}$ and $\theta(c_2) = \frac{n}{2}.$ Hence $\chi'(C^x_{n,c}) = 1\cdot\theta(c_1) + 2\cdot\theta(c_2) = \frac{n}{2} + 2\cdot\frac{n}{2} = \frac{3n}{2} = min\{\chi'(C^{x_i}_{n,c}): \forall C^{x_i}_{n,c}\}.$\\
Subcase (ii)(a)(2): If $\frac{n}{2}$ is even, consider $C^x_{n,c} = C_n + (v_1v_{\frac{n}{2}}, v_{\frac{n}{2}+1}v_n, v_{1+i}v_{n-i})$, $1\leq i \leq \frac{n}{2}-1$. The result now follows like in Subcase (ii)(a)(1).\\
Case (ii)(b): Consider $C_n$, $n\geq 6.$\\
Subcase (ii)(b)(1): If $\frac{n}{2}$ is even let $t=\frac{n}{2}.$ Construct the graph $C^*_n = C_n + (v_1v_{t+1},v_{t-1}v_{t+3},v_tv_{t+2})$. Now allocate the proper coloring $v_1\rightarrow c_3, v_2\rightarrow c_2, v_3\rightarrow c_3,\dots, v_{t-1}\rightarrow c_2, v_t \rightarrow c_3, v_{t+1}\rightarrow c_1, v_{t+2}\rightarrow c_2, v_{t+3}\rightarrow c_3,\dots, v_{n-1}\rightarrow c_2, v_n\rightarrow c_3.$ Since equal number of non-adjacent vertices are left with degree 2 and half colored $c_2$ and other half colored $c_3$ respectively, it is always possible to complete the extension to obtain $C^x_{n,c}$ having the allocated minimum proper coloring. Clearly $\theta(c_1) =1,$ $\theta(c_2) = t-1$ and $\theta(c_3) = t.$ Hence, $\chi^+(C^x_{n,c}) = 3\cdot\frac{n}{2} + 2\cdot(\frac{n}{2} -1) + 1 = \frac{5n}{2} -1 = max\{\chi^+(C^{x_i}_{n,c}):\forall C^{x_i}_{n,c}\}.$\\
Subcase (ii)(b)(2): If $\frac{n}{2}$ is odd let $t=\frac{n}{2}.$ Construct the graph $C^*_n = C_n + (v_1v_{t+1},v_{t-1}v_{t+3},v_tv_{t+2})$. Now allocate the proper coloring $v_1\rightarrow c_3, v_2\rightarrow c_2, v_3\rightarrow c_3,\dots, v_{t-1}\rightarrow c_2, v_t \rightarrow c_3, v_{t+1}\rightarrow c_1, v_{t+2}\rightarrow c_2, v_{t+3}\rightarrow c_3,\dots, v_{n-1}\rightarrow c_1, v_n\rightarrow c_2.$ Since equal number of non-adjacent vertices are left with degree 2 and half colored $c_2$ and other half colored $c_3$ respectively, it is always possible to complete the extension to obtain $C^x_{n,c}$ having the allocated minimum proper coloring. Clearly $\theta(c_1) =2,$ $\theta(c_2) = t-1$ and $\theta(c_3) = t-1.$ Hence, $\chi^+(C^x_{n,c}) = 3\cdot(\frac{n}{2}-1) + 2\cdot(\frac{n}{2} -1) + 2 = \frac{5n}{2} -3 = max\{\chi^+(C^{x_i}_{n,c}):\forall C^{x_i}_{n,c}\}.$\\
\end{proof}
\subsection{On Bipartite Graphs}
It is well known that a graph $G$ is bipartite if and only $G$ contains no odd cycle. Up to equivalence a graph (connected) which contains no odd cycle has a unique vertex-set partition say $X,Y$ such that if $v,u \in X$ then $vu \notin E(G)$ and similarly for vertices in $Y$. If $|X|$ and $|Y|$ are even we say the partition is \emph{even balanced}, else if $|X|$ and $|Y|$ are odd it is \emph{odd balanced}. We also say ordered pairs $(a,b)$ and $(c,d)$ are \emph{bi-distinct} if and only if $a\neq c$ and $b\neq d.$

\begin{theorem}
If $G$ is isomorphic to a balanced bipartite graph $B_{n,m}$ with $n \geq m$ then $G$ has a complete extended graph $G^x_c$.
\end{theorem}
\begin{proof}
Case 1: Consider the even balanced bipartite graph $B_{n,m}$ which is isomorphic to $G$. Without loss of generality assume $|X| =n \geq m = |Y|$. Label the vertices in $X$, $v_1,v_2,v_3\dots,v_n$ and those in $Y$, $u_1,u_2,u_3,\dots,u_m.$ Identify the maximum number of  bi-distinct non-adjacent vertex pairs $(v_i,u_j)$ in $G$. Assume there are $\ell$ such bi-distinct pairs.\\
Subcase 1(i): If $\ell$ is even, add the edges $v_iu_j$ for all $\ell$ bi-distinct pairs. Clearly $n-\ell$ vertices in $X$ can be paired into bi-distinct vertex pairs. Add those corresponding pairwise edges. Similarly, $n-\ell$ vertices in $Y$ can be paired into bi-distinct vertex pairs. Add those corresponding pairwise edges as well. The new graph obtained through this construction is a complete degree-extension $G^x_c$ of $G$.\\
Subcase 1(ii): If $\ell$ is odd select any $\ell -1$ bi-distinct vertex pairs. Construct a new graph similar to Case 1(i) which is clearly $G^x_c$.\\
Case 2: Consider the odd balanced bipartite graph $B_{n,m}$ which is isomorphic to $G$. As before assume $|X| = n \geq m = |Y|$ and label the vertices similarly.\\
Subcase 2(i): If If $\ell$ is even select any $\ell -1$ bi-distinct vertex pairs. Construct a new graph similar to Case 1(i) which is clearly $G^x_c$.\\
Subcase 2(ii): If $\ell$ is even, construct a new graph similar to Case 1(ii) which is clearly $G^x_c$.\\
\end{proof}
We note that in all cases above the subgraph $G-\{v_i,u_j:\forall \ell$ or $\ell-1,$ $(v_i,u_j)$ is a bi-distinct non-adjacent pair of vertices$\}$ is a complete bipartite subgraph.
\begin{theorem}
If $G$ is isomorphic to a balanced bipartite graph $B_{n,m}$ with $|X|=n \geq m = |Y|$ and $G$ has a maximum $\ell$ bi-distinct pairs of vertices $(v_i,u_j)$, $v_i \in X$, $u_j \in Y$, then:\\
Case 1: If both $n$, $m$ are even:\\
(i) If $\ell$ is even, $\chi'(G^x_c) = \frac{3(n-\ell)}{2} +\frac{7(m-\ell)}{2} + 3\ell$.\\
(ii) If $\ell$ is odd, $\chi'(G^x_c) = \lfloor\frac{n-\ell}{2}\rfloor + 2\cdot\lceil\frac{n-\ell}{2}\rceil + \frac{7(m-\ell+1)}{2} + 3\ell$.\\\\
Case 2: If both $n$, $m$ are odd:\\
(i) If $\ell$ is even, $\chi'(G^x_c) = \lfloor\frac{n-\ell}{2}\rfloor + 2\cdot\lceil\frac{n-\ell}{2}\rceil + \frac{7(m-\ell+1)}{2} + 3\ell$.\\
(ii) If $\ell$ is odd, $\chi'(G^x_c) = \frac{3(n-\ell)}{2} +\frac{7(m-\ell)}{2} + 3\ell$.
\end{theorem}
\begin{proof}
Case 1: Assume both $|X|$ and $|Y|$ are even. Label the vertices in $X$, $v_1,v_2,v_3\dots,v_n$ and those in $Y$, $u_1,u_2,u_3,\dots,u_m.$\\
Subcase 1(i): Let $\ell$ be even. Without loss of generality label the vertices of $X,Y$ which belong to the $\ell$ bi-distinct non-adjacent pairs as $v_1, v_2, v_3,\dots, v_\ell$ and $u_1,u_2,u_3,\dots,u_\ell$, respectively such that $(v_i,u_j)$, $1\leq i \leq \ell$ are non-adjacent. Add edges $v_iu_i$, $1\leq i \leq \ell$ and assign the colors $v_1\rightarrow c_1, u_i\rightarrow c_2$, $1\leq i \leq \ell.$ Label the rest of the vertices of $X$ and $Y$ to be $v_{\ell+1},v_{\ell+2}\dots,v_n$ and $u_{\ell+1},u_{\ell+2},\dots,u_m$, respectively. Now add the edges $v_{\ell+1}v_{\ell+2},u_{\ell+1}u_{\ell+2}$ and assign the colors $v_{\ell+1}\rightarrow c_1,v_{\ell+2}\rightarrow c_2, u_{\ell+1}\rightarrow c_3, u_{\ell+2} \rightarrow c_4$. Note that, $v_{\ell+1}, v_{\ell+2}, u_{\ell+1}, u_{\ell+1}$ forms a complete graph of order $4$ ($K_4$) in $G^x_c$ since we have chosen $\ell$ to be maximum.  Recursively proceed with the edge-adding and coloring protocols until only the $(n-m)$ vertices $v_{m+1}, v_{m+2},\dots, v_n$ are left. Add the edges $v_{m+1}v_{m+ 2}, v_{m+3}v_{m+ 4},\dots v_{n-1}v_n,$ and assign the colors, alternating $v_{m+1}\rightarrow c_1, v_{m+2}\rightarrow c_2,\dots v_{n-1}\rightarrow c_1, v_n\rightarrow c_2.$ Clearly the assignment of colors is a minimum proper coloring of $G^x_c$.\\\\
It follows that $\theta(c_1) = \frac{n-\ell}{2} +\ell,$ $\theta(c_2) = \frac{n-\ell}{2} +\ell,$ $\theta(c_3) = \frac{m-\ell}{2},$ $\theta(c_4) = \frac{m-\ell}{2}.$\\
$\therefore \chi'(G^x_c) = min\{\sum\limits_{i=1}^{k}i\cdot\theta(c_i): \forall$  minimum proper colorings of $G^x\} = \theta(c_1) + 2\cdot\theta(c_2) + 3\cdot\theta(c_3) +\\ 4\cdot\theta(c_4) = \frac{3(n-\ell)}{2} +\frac{7(m-\ell)}{2} + 3\ell$.\\
Subcase 1(ii): Follows similar to Case (i) by considering only $\ell-1$ of the maximum bi-distinct non-adjacents pairs of vertices.\\\\
Case 2: Assume both $|X|$ and $|Y|$ are odd.\\
Subcase 2(i): Let $\ell$ be even and label the vertices as before. The proof follows from similar reasoning found in Subcase 1(ii).\\
Subcase 2(ii): Let $\ell$ be odd and label the vertices as before. The proof follows from similar reasoning found in Subcase 1(i).
\end{proof}
 Note that the maximum matching in a bipartite graph hence, the maximum bi-distinct non-adjacent pairs of vertices and the value $\ell$ can be determined by amongst others, the $n^{\frac{5}{2}}$-Algorithm described by Hopcroft and Karp in [6]. We now discuss a special case of this theorem where $|X|=|Y|$, which means both the partitions of the bipartite graph have equal cardinality. This class of graph inclues all regular bipartite graphs. For this case we calculate both   $\chi'(G^x)$ and $\chi^{+}(G^x_c)$.

\begin{definition}
Let $G$ be a bipartite graph, with vertex partitions $X$ an $Y$. Let $S$ be a subset of $X$. For, $v \in S$  we define $N^c(v)$ to be the set of vertices in $Y$ that are not adjacent to $v$. We define $N^c(S)$ to be $\bigcup_{v \in S} N^c(v)$.
\end{definition}

\begin{corollary}
Let $G$ be a bipartite graph of even order $n$, with vertex partitions $X$ an $Y$, such that $|X|=|Y|$. If for every $S \subseteq X$, $|N^c(S)| \geq |S|$ then,\\
(i) $\chi'(G^x_c)=\frac{3n}{2}$.\\
(ii) $\chi^{+}(G^x_c)= \frac{5n}{2}-3$.
\end{corollary}
\begin{proof}
Case (i): Note that in $G$, $X$ and $Y$ are independent sets. Hence, in $G^c$ vertices in $\langle X\rangle$ and $\langle Y\rangle$ will be cliques of size $|X|$ and $|Y|$, respectively. Let us denote these cliques as $K_{|X|}$ and $K_{|Y|}$ respectively. Also, note that the graph $G^c- E(K_{|X|}) - E(K_{|Y|})$ is bipartite. By Hall's Theorem the graph $G^c- E(K_{|X|}) - E(K_{|Y|})$ has a perfect matching $M$, since in $G$, for every $S \subseteq X$, $|N^c(S)| \geq |S|$. Thus, $G$ can always be extended to a bipartite graph. Hence, $\chi'(G^x_c)=\frac{3n}{2}$.\\
Case (ii): Now, let us assign a different extension of $G$. Choose two arbitrary edges from $M$, $x_1y_1$ and $x_2y_2$, such that $x_1, x_2 \in X$ and $y_1,y_2 \in Y$. Now, for extending $G$, use the edges $M-x_1y_1-x_2y_2+x_1x_2+y_1y_2$. This can always be done since the graph is bipartite. Now, we can color $x_1$ and $y_1$ using $c_1$.  Vertices in $X-x_1$ and $Y-y_1$ can be colored using $c_2$ and $c_3$, respectively. This obviously shows that, $\chi^{+}(G^x_c)= \frac{5n}{2}-3$.
\end{proof}

\section{On some specific classes of Graphs}

\begin{lemma}
Let $G$ be a graph of order $n$, such that $\delta(G)> \frac{n}{2}$, then $diam(G) \leq 2.$
\end{lemma}
\begin{proof}
Let $A$ be the adjacency matrix of $G$. Every entry $a_{i,j}$ of $A^2$ is the scalar product of the $i$-th row and the $j$-th column of $A$. Since every row and column of $A$ has at least $\frac{n}{2}$ entries equal to 1, every $a_{i,j} > 0$, $a_{i,j}$ an entry of $A^2.$ Hence for any two distinct vertices $v,u \in V(G)$ we have $d_G(v,u) \leq 2.$ Therefore $diam(G) \leq 2.$
\end{proof}
\begin{theorem}
Let $G$ be a graph of order $n$, such that $\delta(G)> \frac{n}{2}$. Then $G$ is bipartite if and only it is triangle free and $K_{r,r}$ is the only such graph.
\end{theorem}
\begin{proof}
Necessary condition: It follows from Lemma 4.1 that $diam(G) \leq 2.$ Therefore the largest holes $G$ can have are of the kind, $C_4$. If $G$ is triangle free it has no odd cycle hence, $G$ is bipartite.\\
Sufficient condition: If $G$ is bipartite it contains no odd cycle hence, $G$ is triangle free.\\
Finally $K_{r,r}$ is the only $r$-regular bipartite graph on exactly (at most) $2r$ vertices.
\end{proof}
We state the next lemma without proof.
\begin{lemma}
Let graphs $G$ and $H$ have minimum proper coloring sets $\mathcal{C}_G =\{c_1,c_2,c_3,\dots,c_{\chi(G)}\}$ and $\mathcal{C}_H =\{c_1,c_2,c_3,\dots,c_{\chi(H)}\}$, respectively. Assume $\chi(G)\leq \chi(H)$. Let $v_i\rightarrow c_k$, $v_i\in V(G)$ and $u_j\rightarrow c_t$, $u_j\in V(H)$ and both $c_k,c_t \in \mathcal{C}_G$ hence, $c_t \in \mathcal{C}_H,$ as well.\\
(i) If $G$ and $H$ are joined by merging vertices $v_i,u_j$ as a common vertex which is now assigned the color $c_k$; an equivalent minimum proper coloring of $H$ is possible by re-assigning color $c_k$ to all vertices in $H$ which were assigned $c_t$; and assigning color $c_t$ to all vertices in $H$ which were assigned color $c_k.$\\
(ii) If graphs $G$ and $H$ are allowed to join by merging an edge each into a common edge, result (i) holds by applying it consecutively to the pairs of vertices $(v_i,u_j)$ and $(v_\ell, u_m)$, $v_iv_\ell \in E(G)$ and $u_ju_m \in E(H).$
\end{lemma}
Now we introduce a family of \emph{pattern structured} graphs. When a cluster of two copies of a graph $H$ are allowed to merge at least one edge (not necessarily structurally equivalent edges) to share at least one common edge, the new graph is called a $H$-\emph{gridlike cluster}. Two or more $H$-\emph{gridlike clusters} are allowed to merge similarly to get an expanded $H$-\emph{gridlike cluster}. When a cluster of two or more copies of a graph $H$ are all allowed to merge a vertex (not necessarily structurally equivalent vertices) to share a common vertex, the new graph is called a $H$-\emph{cloverlike cluster}. When a cluster of two or more copies of a graph $H$ are all allowed to merge an edge (not necessarily structurally equivalent edges) to all share a common edge, the new graph is called a $H$-\emph{booklike cluster}. When two copies $H_1, H_2$ are joined by a path $ve_0w_1e_1w_2e_2w_3\dots e_{m-1}w_me_mu$, $v \in V(H_1),u \in V(H_2)$ we say they are adjacent. If a graph $G^*$ is the composition of the aforesaid then; a vertex, a $H$-gridlike cluster, a $H$-cloverlike cluster, a $H$-booklike cluster are called $H$-elements of $G^*$. If $G^*$ has a cycle between between at least two $H$-elements $G^*$ is called cyclic, else it is called acyclic or $H$-\emph{treelike}.
\begin{theorem}
A $H$-treelike graph $G^*$ with $\chi(H) \geq 2$ has, $\chi(G^*) = \chi(H).$
\end{theorem}
\begin{proof}
Consider any $H$-treelike graph $G^*$.\\
Case (i): Consider a $H$-gridlike element and without loss of generality consider any pair $H_1,H_2$ sharing a common edge say $uv$. If in a minimum proper coloring of $H_1$ the vertex coloring is $v\rightarrow c_i, u\rightarrow c_j$ then after applying Lemma 4.3 the $\chi$-number of the partial $H$-gridlike element remains the same.\\
Case (ii): Consider a $H$-cloverlike element and without loss of generality assume the common vertex is $v$. If in a minimum proper coloring of $H_1$ the vertex coloring is $v\rightarrow c_i$ then after applying Lemma 4.3 (i) the $\chi$-number in the partial $H$-cloverlike element remains the same. By iteratively applying Lemma 4.3 (i) to all pairs, the $\chi$-number of the whole $H$-cloverlike element remains the same.\\
Case (iii): Consider a $H$-booklike element. Similar reasoning as in Case (i) follows.\\
Case (iv): First consider adjacent $H$-elements $H_1$ and $H_2$ as the vertex join between $H_1$ and an end vertex of a path and apply Lemma 4.3 (i). Thereafter, consider the vertex join between the other end vertex of the path and $H_2$ and apply Lemma 4.3 (i) again. Clearly the $\chi$-number remains the same.\\\\
Invoking Cases (i) to (iv) throughout a $H$-treelike graph $G^*$ settles the result in general. 
\end{proof}
\section{Conclusion}
We recall that an \emph{almost regular} graph $G$ is such that $\Delta(G) - \delta(G) = 1.$ A \emph{partial extension} of an almost regular graph $G$ is defined to be the graph $G^{px}$ obtained by adding edges to $G$ such that $G^{px}$ is $\Delta(G)$-regular. A path is such an almost regular graph so it is easy to see that $P^{px}_n = C_n$, therefore $\chi'(P^{px}_n) = \chi'(C_n)$ and $\chi^+(P^{px}_n) = \chi^+(C_n)$. There is scope to research $\chi'(G^{px}_n)$ and $\chi^+(G^{px}_n)$ for other classes of almost regular graphs.\\\\
Kouider and El Sahili [7] showed that for a r-regular graph $G$ on at least $r^4$ vertices the $b$-chromatic number is $\varphi(G) = r+1.$ In a paper by Cabello and Jakovac [3] they improved the result by bounding the result to graphs having at least $2r^3$ vertices. So it is easy to see that for $C_n$, $n \geq 54$ we have $\varphi(C_n) =3$ and $\varphi(C^x_{n,c}) = 4.$ In general it will be worthy to research the relationship, if any, between $\varphi'(C_n)$ and $\varphi'(C^x_{n,c})$ as well as that between $\varphi^+(C_n)$ and $\varphi^+(C^x_{n,c})$ and perhaps for other classes of regular graphs.
\section{Acknowledgement}
The authors express their sincere gratitude to Professor Anirban Banerjee, Department of Mathematics and Statistics, Indian Institute of Science Education, Kolkata, India, for his constructive and insightfull comments. The authors also wish Professor Banerjee the complete return of great health. We need his amazing intellectual capacity to guide many of us for many years to come.\\\\
\textbf{\emph{Open access:}} This paper is distributed under the terms of the Creative Commons Attribution License which permits any use, distribution and reproduction in any medium, provided the original author(s) and the source are credited. \\ \\
References (Limited) \\ \\
$[1]$  A. Banerjee and S. Bej, \emph{On Extension of Regular Graphs,} arXiv:1509.05476v1, [math.CO], 17 September 2015. \\
$[2]$  J.A. Bondy and U.S.R. Murty, \textbf {Graph Theory with Applications,} Macmillan Press, London, (1976). \\
$[3]$  S. Cabello and M. Jacovac, \emph{On the b-chromatic number of regular graphs,} Discrete Applied Mathematics, \textbf{159}, (2011), pp 1303-1310. \\
$[4]$ G. Chartrand and L. Lesniak, \textbf{Graphs and Digraphs}, CRC Press, 2000.\\
$[5]$ J.T. Gross and J. Yellen, \textbf{Graph Theory and its Applications}, CRC Press, 2006.\\
$[6]$ J.E. Hopcroft and R.M. Karp, \emph{An $n^{\frac{5}{2}}$ Algorithm for Maximum Matchings in Bipartite Graphs}, SIAM Journal on Computing, 1973, \textbf{2}(4), pp 225-231.\\
$[7]$ J. Kok and N.K. Sudev, \emph{General Coloring Sums of Certain Graphs}, preprint.\\
$[8]$ M. Kouider and A. El Sahili, \emph{About b-coloring of regular graphs}, Rapport de Recherche, No 1432, CNRS-Universite Paris Sud-LRI.\\
$[9]$ P.C. Lisna and M.S. Sunitha, \emph{b-Chromatic Sum of a Graph}, Discrete Mathematics, Algorithms and Applications, \textbf{7}(3), (2015), pp 1550040(15).\\
\end{document}